\documentclass[oneside, reqno]{amsart}

\usepackage{amsaddr}

\usepackage[margin=1in]{geometry}

\usepackage{amsthm,amssymb,amsmath,graphicx}
\usepackage{bm}
\allowdisplaybreaks

\usepackage[usenames, dvipsnames]{color}
\usepackage{hyperref}
\hypersetup{
  colorlinks=true,
  urlcolor=blue,
  linkcolor=MidnightBlue,
  citecolor=ForestGreen,
  pdfinfo={
    CreationDate={D:20180822233421},
    ModDate={D:20180822233421},
  },
}

\usepackage{embedfile}
\embedfile{\jobname.tex}

\usepackage{chngcntr}
\counterwithin{figure}{section}
\counterwithin{equation}{section}
\usepackage[labelfont=bf]{caption}

\theoremstyle{definition}
\newtheorem{theorem}{Theorem}[section]

\newcommand{\reals}{\mathbf{R}}

\newcommand{\bigO}[1]{\mathcal{O}\left(#1\right)}
\newcommand{\mach}{\mathbf{u}}
\newcommand{\utri}{\mathcal{U}}

\newcommand{\rev}[1]{#1}

\begin{document}

\title{High-order Solution Transfer between Curved Triangular Meshes}

\author{Danny Hermes}
\email{dhermes@berkeley.edu}

\author{Per-Olof Persson}
\address{UC Berkeley, 970 Evans Hall \#3840, Berkeley, CA 94720-3840 USA}
\email{persson@berkeley.edu}

\begin{abstract}
\noindent The problem of solution transfer between meshes arises frequently in
computational physics, e.g. in Lagrangian methods where remeshing
occurs. The interpolation process must be conservative, i.e. it
must conserve physical properties, such as mass. We extend previous
works --- which described the solution transfer process for straight sided
unstructured meshes --- by considering high-order isoparametric meshes
with curved elements. To facilitate solution transfer, we numerically
integrate the product of shape functions via Green's theorem along the
boundary of the intersection of two curved elements. We perform a numerical
experiment and confirm the expected accuracy by transferring test fields
across two families of meshes.
\end{abstract}

\maketitle

\section{Introduction}

Solution transfer between meshes is a common problem that
has practical applications to many methods in
computational physics. \rev{For example, by allowing the underyling
computational domain to change during a simulation, computational
effort can be focused dynamically to resolve relevant features
of a numerical solution. This so-called mesh adaptivity typically
requires translating the numerical solution from the old mesh to the new,
\rev{i.e. solution transfer \cite{Babuska1978,Peraire1987,Pain2001,Iske2004}}}. As another example, Lagrangian or particle-based
methods treat each node in the mesh as a particle and so with each timestep the
\rev{mesh travels \emph{with} the fluid \cite{MR3023731}.
However, over (typically short) time the mesh
becomes distorted and suffers a loss in element quality which eventually causes
inaccurate solutions.} To overcome this, the
domain must be remeshed or rezoned and the solution must be
transferred (remapped) onto the new mesh configuration. \rev{A related more general
class of methods is the Arbitrary Lagrangian-Eulerian (ALE) method, which
are not constrained by a fixed computational mesh (the Eulerian approach) or
by a fixed fluid flow (the Lagrangian approach) and typically combine the benefits
of both approaches \cite{Hirt1974,MR3302403}.}

Solution transfer is needed when a solution
(approximated by a discrete field) is known on a \emph{donor} mesh and must
be transferred to a \emph{target} mesh.
When pointwise interpolation is used to transfer a solution, quantities with
physical meaning (e.g. mass, concentration, energy) may not be conserved.
In many applications, the field
must be conserved for physical reasons, e.g. mass or energy cannot leave or
enter the system, hence this paper focuses on \emph{conservative} solution transfer
(typically using Galerkin or \(L_2\)-minimizing methods).

The problem of conservative interpolation
has been considered previously for straight sided meshes.
However, both to allow for greater
geometric flexibility and for high order of convergence, this paper
addresses the case of curved isoparametric meshes. \rev{This is based on the
recent interest in so-called high-order methods \cite{Wang2013}, which have the ability to
produce highly accurate solutions with low dissipation and low dispersion error.
However, these methods typically require curved meshes to obtain the high accuracy.}

The \emph{common refinement} approach in
\cite{Jiao2004} is used to compare several methods for solution transfer across
two meshes. The problem of constructing such a refinement is
considered in \cite{Farrell2009, Farrell2011} (called a supermesh by
the authors).
However, the solution transfer becomes considerably more challenging for curved
meshes. For a sense of the difference between the straight sided and curved
cases, consider the problem of intersecting an element from the donor mesh
with an element from the target mesh. If the elements are triangles, the
intersection is either a convex polygon or has measure zero. If the elements
are curved, the intersection can be non-convex and can even split into
multiple disjoint regions. In \cite{Qiu2018}, a related problem is considered
where squares are allowed to curve along characteristics and the solution is
remapped onto a regular square grid.

Conservative solution transfer has been around since the advent of ALE,
and as a result much of the existing literature focuses on mesh-mesh
pairs that will occur during an ALE-based simulation. \rev{Many ALE methods
modify the mesh during simulation: when flow-based
mesh distortion occurs, elements are typically ``flipped'' (e.g. a
diagonal is switched in a pair of elements) or elements are subdivided
or combined.} These operations are inherently local, hence the solution
transfer can be done locally across known neighbors. Typically, this
locality is crucial to solution transfer methods. In \cite{Margolin2003},
the transfer is based on partitioning elements of the updated mesh into
components of elements from the old mesh and ``swept regions'' from
neighbouring elements. In \cite{Kucharik2008}, the (locally) changing
connectivity of the mesh is addressed. In \cite{Garimella2007}, the
local transfer is done on polyhedral meshes.

Global solution transfer instead seeks to conserve the solution across
the whole mesh. It makes no assumptions about the relationship between
the donor and target meshes. The loss in local information makes
the mesh intersection problem more computationally expensive, but the
added flexibility reduces timestep restrictions since it allows remeshing
to be done less often. In \cite{Dukowicz1984, Dukowicz1987}, a global
transfer is enabled by transforming volume integrals to surface integrals
via the divergence theorem to reduce the complexity of the problem.

\rev{In addition to conservation, there are many other considerations for a
solution transfer method that are important in various applications. For
example, since our focus is on high-order methods, these typically produce
oscillations in the solutions that can be non-physical and break the numerical
solvers. This can happen for the standard Galerkin projections that we
consider here, which means that e.g. discontinuities such as shocks in the
solution could become nonphysical after our solution transfer. Various
solutions to this have been proposed before, often referred to as limiting or
bound/monotonicity preserving methods, see e.g. \cite{MR3302403,Anderson2018}
where several approaches are studied including a high-order extension of the
flux-corrected transport (FCT) method. Other considerations for a solution
transfer scheme include the preservation of additional constraints, such
as the divergence or the curl of the solution.}

In this paper, an algorithm for conservative solution transfer between curved
meshes will be described. This method applies to meshes in \(\reals^2\).
Application to meshes in \(\reals^3\) is a direction for future research,
though the geometric kernels (see Chapter~\ref{sec:bezier-intersection})
become significantly more challenging to describe and implement. In
addition, the method will assume that every element in the
target mesh is contained in the donor mesh. This ensures that the solution
transfer is \rev{\emph{interpolative}}. In the case where all target elements are
partially covered, \emph{extrapolation} could be used to extend a solution
outside the domain, but for totally uncovered elements there is no clear
correspondence to elements in the donor mesh.

This paper is organized as follows. Section~\ref{sec:preliminaries}
establishes common notation and reviews basic results relevant to the
topics at hand. Section~\ref{sec:bezier-intersection} is an
in-depth discussion of the computational geometry methods needed
to implement to enable solution transfer. Section~\ref{sec:galerkin-projection}
(and following)
describes the solution transfer process and gives results of some
numerical experiments confirming the rate of convergence.

\section{Preliminaries}\label{sec:preliminaries}

\subsection{General Notation}

We'll refer to \(\reals\) for the reals, \(\utri\) represents
the unit triangle (or unit simplex) in \(\reals^2\):
\begin{equation}
\utri = \left\{(s, t) \mid 0 \leq s, t, s + t \leq 1\right\}.
\end{equation}
When dealing with sequences with multiple indices, e.g.
\(s_{m, n} = m + n\), we'll use bold symbols to represent
a multi-index: \(\bm{i} = (m, n)\).
The binomial coefficient
\(\binom{n}{k}\) is equal to \(\frac{n!}{k! (n - k)!}\) and the trinomial
coefficient \(\binom{n}{i, j, k}\) is equal to \(\frac{n!}{i! j! k!}\)
(where \(i + j + k = n\)). The notation \(\delta_{ij}\) represents the
Kronecker delta, a value which is \(1\) when \(i = j\) and \(0\)
otherwise.

\subsection{B\'{e}zier Curves}

A \emph{B\'{e}zier curve} is a mapping from the unit interval
that is determined by a set of control points
\(\left\{\bm{p}_j\right\}_{j = 0}^n \subset \reals^d\).
For a parameter \(s \in \left[0, 1\right]\), there is a corresponding
point on the curve:
\begin{equation}
b(s) = \sum_{j = 0}^n \binom{n}{j} (1 - s)^{n - j} s^j \bm{p}_j \in
  \reals^d.
\end{equation}
This is a combination of the control points weighted by
each Bernstein basis function
\(B_{j, n}(s) = \binom{n}{j} (1 - s)^{n - j} s^j\).
Due to the binomial expansion
\(1 = (s + (1 - s))^n = \sum_{j = 0}^n B_{j, n}(s)\),
a Bernstein basis function is in
\(\left[0, 1\right]\) when \(s\) is as well. Due to this fact, the
curve must be contained in the convex hull of \rev{its} control points.

\subsection{B\'{e}zier Triangles}

A \emph{B\'{e}zier triangle} (\cite[Chapter~17]{Farin2001}) is a
mapping from the unit triangle
\(\utri\) and is determined by a control net
\(\left\{\bm{p}_{i, j, k}\right\}_{i + j + k = n} \subset \reals^d\).
A B\'{e}zier triangle is a particular kind of B\'{e}zier surface, i.e. one
in which there are two cartesian or three barycentric input parameters.
Often the term B\'{e}zier surface is used to refer to a tensor product or
rectangular patch.
For \((s, t) \in \utri\) we can define barycentric weights
\(\lambda_1 = 1 - s - t, \lambda_2 = s, \lambda_3 = t\) so that
\begin{equation}
1 = \left(\lambda_1 + \lambda_2 + \lambda_3\right)^n =
  \sum_{\substack{i + j + k = n \\ i, j, k \geq 0}} \binom{n}{i, j, k}
  \lambda_1^i \lambda_2^j \lambda_3^k.
\end{equation}
Using this we can similarly define a (triangular) Bernstein basis
\begin{equation}
B_{i, j, k}(s, t) = \binom{n}{i, j, k} (1 - s - t)^i s^j t^k
  = \binom{n}{i, j, k} \lambda_1^i \lambda_2^j \lambda_3^k
\end{equation}
that is in \(\left[0, 1\right]\) when \((s, t)\) is in \(\utri\).
Using this, we define points on the B\'{e}zier triangle as a
convex combination of the control net:
\begin{equation}
b(s, t) = \sum_{i + j + k = n} \binom{n}{i, j, k}
  \lambda_1^i \lambda_2^j \lambda_3^k
  \bm{p}_{i, j, k} \in \reals^d.
\end{equation}

\begin{figure}
  \includegraphics{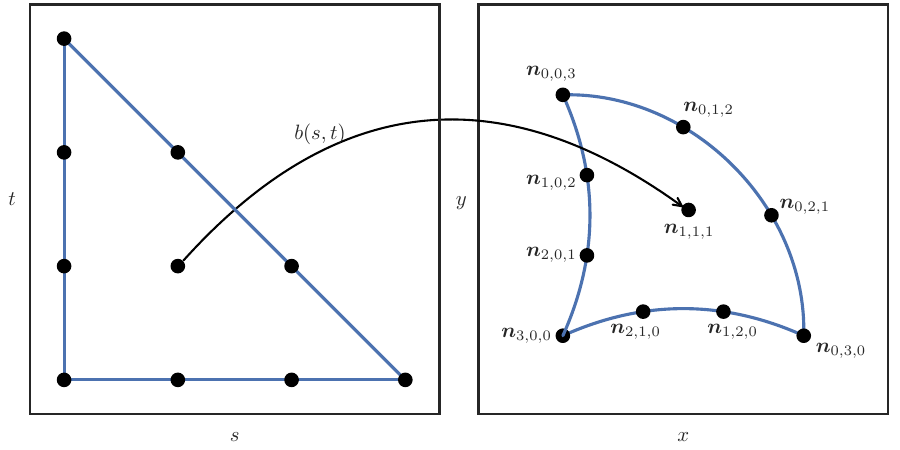}
  \centering
  \captionsetup{width=.75\linewidth}
  \caption{Cubic B\'{e}zier triangle}
  \label{fig:cubic-bezier-example}
\end{figure}

\noindent Rather than defining a B\'{e}zier triangle by the control net, it can
also be uniquely determined by the image of a standard lattice of
points in \(\utri\): \(b\left(j/n, k/n\right) = \bm{n}_{i, j, k}\);
we'll refer to these as \emph{standard nodes}.
Figure~\ref{fig:cubic-bezier-example} shows these standard nodes for
a cubic triangle in \(\reals^2\). To see the correspondence,
when \(p = 1\) the standard nodes \emph{are} the control net
\begin{equation}
b(s, t) = \lambda_1 \bm{n}_{1, 0, 0} +
\lambda_2 \bm{n}_{0, 1, 0} + \lambda_3 \bm{n}_{0, 0, 1}
\end{equation}
and when \(p = 2\)
\begin{multline}
b(s, t) = \lambda_1\left(2 \lambda_1 - 1\right) \bm{n}_{2, 0, 0} +
\lambda_2\left(2 \lambda_2 - 1\right) \bm{n}_{0, 2, 0} +
\lambda_3\left(2 \lambda_3 - 1\right) \bm{n}_{0, 0, 2} + \\
4 \lambda_1 \lambda_2 \bm{n}_{1, 1, 0} +
4 \lambda_2 \lambda_3 \bm{n}_{0, 1, 1} +
4 \lambda_3 \lambda_1 \bm{n}_{1, 0, 1}.
\end{multline}
However, it's worth noting that the transformation between
the control net and the standard nodes has condition
number that grows exponentially with \(n\) (see \cite{Farouki1991}, which
is related but does not directly show this).
This may make working with
higher degree triangles prohibitively unstable.

A \emph{valid} B\'{e}zier triangle is one which is
diffeomorphic to \(\utri\), i.e. \(b(s, t)\) is bijective and has
an \rev{invertible Jacobian everywhere}. We must also have the orientation
preserved, i.e. the Jacobian must have positive determinant. For example, in
Figure~\ref{fig:inverted-element}, the image of \(\utri\) under
the map \(b(s, t) = \left[\begin{array}{c c} (1 - s - t)^2 + s^2 & s^2 + t^2
\end{array}\right]^T\) is not valid because the Jacobian is zero along
the curve \(s^2 - st - t^2 - s + t = 0\) (the dashed line). Elements that
are not valid are called \emph{inverted} because they have regions with
``negative area''. For the example, the image \(b\left(\utri\right)\)
leaves the boundary determined by the edge curves: \(b(r, 0)\),
\(b(1 - r, r)\) and \(b(0, 1 - r)\) when \(r \in \left[0, 1\right]\).
This region outside the boundary is traced twice, once with
a positive Jacobian and once with a negative Jacobian.
\begin{figure}
  \includegraphics{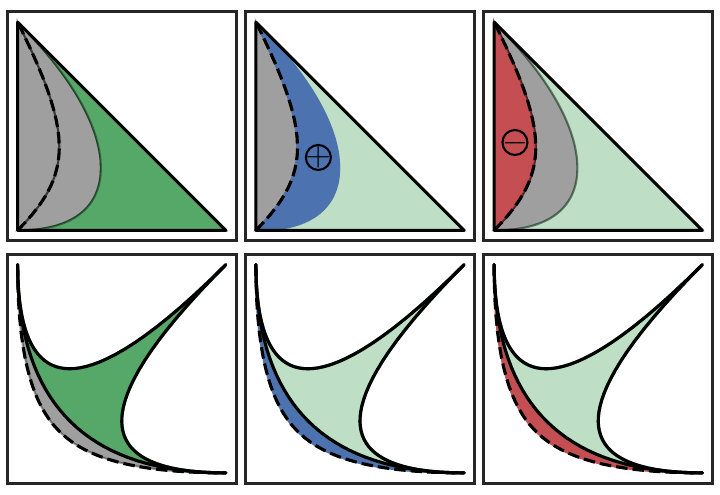}
  \centering
  \captionsetup{width=.75\linewidth}
  \caption{The B\'{e}zier triangle given by \(b(s, t) = \left[
    (1 - s - t)^2 + s^2 \; \; s^2 + t^2 \right]^T\) produces an
    inverted element. It traces the same region twice, once with
    a positive Jacobian (the middle column) and once with a negative
    Jacobian (the right column).}
  \label{fig:inverted-element}
\end{figure}

\subsection{Curved Elements}\label{sec:curved-elements}

We define a curved mesh element \(\mathcal{T}\) of degree \(p\)
to be a B\'{e}zier triangle in \(\reals^2\) of the same degree.
We refer to the component functions of \(b(s, t)\) (the map that
gives \(\mathcal{T} = b\left(\utri\right)\)) as \(x(s, t)\) and \(y(s, t)\).

This fits a typical definition (\cite[Chapter~12]{FEM-ClaesJohnson})
of a curved element, but gives a special meaning to the mapping from
the reference triangle. Interpreting elements as B\'{e}zier triangles
has been used for Lagrangian methods where
mesh adaptivity is needed (e.g. \cite{CardozeMOP04}). Typically curved
elements only have one curved side (\cite{McLeod1972}) since they are used
to resolve geometric features of a boundary. See also
\cite{Zlmal1973, Zlmal1974}.
B\'{e}zier curves and triangles have a number of mathematical properties
(e.g. the convex hull property) that lead to elegant geometric
descriptions and algorithms.

Note that a B\'{e}zier triangle can be
determined from many different sources of data (for example the control net
or the standard nodes). The choice of this data may be changed to suit the
underlying physical problem without changing the actual mapping. Conversely,
the data can be fixed (e.g. as the control net) to avoid costly basis
conversion; once fixed, the equations of motion and other PDE terms can
be recast relative to the new basis (for an example, see \cite{Persson2009},
where the domain varies with time but the problem is reduced to
solving a transformed conservation law in a fixed reference configuration).

\subsection{Shape Functions}\label{subsec:shape-functions}

When defining shape functions (i.e. a basis with geometric meaning) on a
curved element there are (at least) two choices. When the degree of the
shape functions is the same as the degree of the function being
represented on the B\'{e}zier triangle,
we say the element \(\mathcal{T}\) is \emph{isoparametric}.
For the multi-index
\(\bm{i} = (i, j , k)\), we define \(\bm{u}_{\bm{i}} =
\left(j/n, k/n\right)\) and the corresponding standard node
\(\bm{n}_{\bm{i}} = b\left(\bm{u}_{\bm{i}}\right)\).
Given these points, two choices for shape functions present
themselves:
\begin{itemize}
  \itemsep 0em
  \item \emph{Pre-Image Basis}:
    \(\phi_{\bm{j}}\left(\bm{n}_{\bm{i}}\right) =
      \widehat{\phi}_{\bm{j}}\left(\bm{u}_{\bm{i}}\right) =
      \widehat{\phi}_{\bm{j}}\left(b^{-1}\left(
      \bm{n}_{\bm{i}}\right)\right)\)
    where \(\widehat{\phi}_{\bm{j}}\) is a canonical basis function
    on \(\utri\), i.e.
    \(\widehat{\phi}_{\bm{j}}\) a degree \(p\) bivariate polynomial and
    \(\widehat{\phi}_{\bm{j}}\left(\bm{u}_{\bm{i}}\right) =
    \delta_{\bm{i} \bm{j}}\)
  \item \emph{Global Coordinates Basis}:
    \(\phi_{\bm{j}}\left(\bm{n}_{\bm{i}}\right) =
    \delta_{\bm{i} \bm{j}}\), i.e. a canonical basis function
    on the standard nodes \(\left\{\bm{n}_{\bm{i}}\right\}\).
\end{itemize}

\noindent For example, consider a quadratic B\'{e}zier triangle:
\begin{gather}
b(s, t) = \left[ \begin{array}{c c}
    4 (s t + s + t) & 4 (s t + t + 1)
  \end{array}\right]^T \\
\Longrightarrow
\left[ \begin{array}{c c c c c c}
    \bm{n}_{2, 0, 0} &
    \bm{n}_{1, 1, 0} &
    \bm{n}_{0, 2, 0} &
    \bm{n}_{1, 0, 1} &
    \bm{n}_{0, 1, 1} &
    \bm{n}_{0, 0, 2}
  \end{array}\right] = \left[ \begin{array}{c c c c c c}
    0 & 2 & 4 & 2 & 5 & 4 \\
    4 & 4 & 4 & 6 & 7 & 8
  \end{array}\right].
\end{gather}
In the \emph{Global Coordinates Basis}, we have
\begin{equation}
\phi^{G}_{0, 1, 1}(x, y) = \frac{(y - 4) (x - y + 4)}{6}.
\end{equation}
For the \emph{Pre-Image Basis}, we need the inverse
and the canonical basis
\begin{equation}
b^{-1}(x, y) = \left[ \begin{array}{c c}
    \frac{x - y + 4}{4} & \frac{y - 4}{x - y + 8}
  \end{array}\right] \quad \text{and} \quad
\widehat{\phi}_{0, 1, 1}(s, t) = 4 s t
\end{equation}
and together they give
\begin{equation}
\phi^{P}_{0, 1, 1}(x, y) = \frac{(y - 4) (x - y + 4)}{x - y + 8}.
\end{equation}
In general \(\phi_{\bm{j}}^P\) may not even be a rational bivariate
function; due to composition with \(b^{-1}\) we can only guarantee that
it is algebraic (i.e. it can be defined as the zero set of polynomials
with coefficients in \(\reals\left[x, y\right]\)).

\subsection{Curved Polygons}\label{subsec:curved-polygons}

\begin{figure}
  \includegraphics{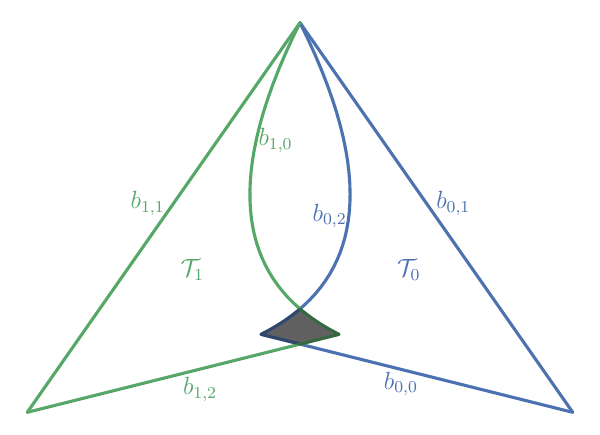}
  \centering
  \captionsetup{width=.75\linewidth}
  \caption{Intersection of B\'{e}zier triangles form a curved polygon.}
  \label{fig:bezier-triangle-intersect}
\end{figure}

When intersecting two curved elements, the resulting surface(s) will
be defined by the boundary, alternating between edges of each
element.
For example, in Figure~\ref{fig:bezier-triangle-intersect}, a
``curved quadrilateral'' is formed when two B\'{e}zier triangles
\(\mathcal{T}_0\) and \(\mathcal{T}_1\) are intersected.

A \emph{curved polygon} is defined by a collection of B\'{e}zier curves
in \(\reals^2\) that determine the boundary. In order to be
a valid polygon, none of the boundary curves may cross, the
ends of consecutive edge curves must meet and the curves must be right-hand
oriented. For our example in
Figure~\ref{fig:bezier-triangle-intersect}, the triangles
have boundaries formed by three B\'{e}zier curves:
\(\partial \mathcal{T}_0 = b_{0, 0} \cup b_{0, 1} \cup b_{0, 2}\) and
\(\partial \mathcal{T}_1 = b_{1, 0} \cup b_{1, 1} \cup b_{1, 2}\).
The intersection \(\mathcal{P}\) is defined by four boundary
curves: \(\partial \mathcal{P} =
C_1 \cup C_2 \cup C_3 \cup C_4\). Each boundary
curve is itself a B\'{e}zier curve\footnote{A specialization of a
B\'{e}zier curve \(b\left(\left[a_1, a_2\right]\right)\)
is also a B\'{e}zier curve.}:
\(C_1 = b_{0, 0}\left(\left[0, 1/8\right]\right)\),
\(C_2 = b_{1, 2}\left(\left[7/8, 1\right]\right)\),
\(C_3 = b_{1, 0}\left(\left[0, 1/7\right]\right)\) and
\(C_4 = b_{0, 2}\left(\left[6/7, 1\right]\right)\).

Though an intersection can be described in terms of the B\'{e}zier triangles,
the structure of the control net will be lost. The region will not in general
be able to be described by a mapping from a simple space like
\(\utri\).

\section{Galerkin Projection}\label{sec:galerkin-projection}

\subsection{Linear System for Projection}

\begin{figure}
  \includegraphics{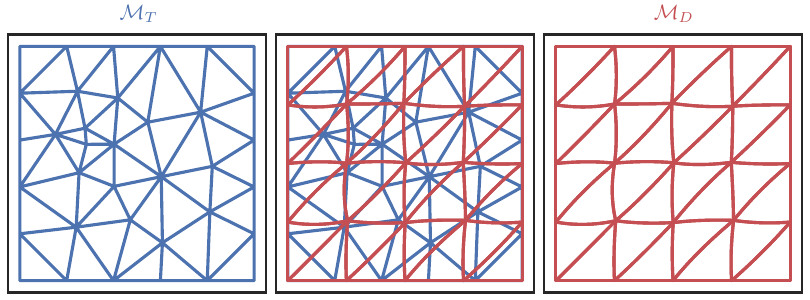}
  \centering
  \captionsetup{width=.75\linewidth}
  \caption{Mesh pair: donor mesh \(\mathcal{M}_D\) and
    target mesh \(\mathcal{M}_T\).}
  \label{fig:donor-target-pair}
\end{figure}

Consider a donor mesh \(\mathcal{M}_D\) with shape function basis
\(\phi_D^{(j)}\) and a known field \(\bm{q}_D = \sum_j d_j \phi_D^{(j)}\)
\footnote{This is somewhat a simplification. \rev{In CG (continuous Galerkin, as
opposed to DG or discontinuous Galerkin)}, some coefficients will be paired
with multiple shape functions, such as the coefficient at a vertex node.}
and a
target mesh \(\mathcal{M}_T\) with shape function basis \(\phi_T^{(j)}\)
(Figure~\ref{fig:donor-target-pair}).
Each shape function \(\phi\) corresponds to a
given isoparametric curved element (see Section~\ref{sec:curved-elements})
\(\mathcal{T}\) in one of these meshes and has
\(\operatorname{supp}(\phi) = \mathcal{T}\).
Additionally, the shape functions are polynomial degree \(p\) (see
Section~\ref{subsec:shape-functions} for a discussion of shape functions),
but the degree of the donor mesh need not be the same as that of the
target mesh. We assume that both meshes cover the same domain \(\Omega \subset
\reals^2\), however we really only require the donor mesh to cover the
target mesh.

We seek the \(L_2\)-optimal interpolant \(\bm{q}_T = \sum_j t_j \phi_T^{(j)}\):
\begin{equation}
\left \lVert \bm{q}_T - \bm{q}_D \right \rVert_2 =
\min_{\bm{q} \in \mathcal{V}_T}
\left \lVert \bm{q} - \bm{q}_D \right \rVert_2
\end{equation}
where \(\mathcal{V}_T = \operatorname{Span}_j\left\{\phi_T^{(j)}\right\}\)
is the function space defined on the target mesh. Since this is optimal in
the \(L_2\) sense, by differentiating with respect to each \(t_j\) in
\(\bm{q}_T\) we find the weak form:
\begin{equation}
\int_{\Omega} \bm{q}_D \phi_T^{(j)} \, dV =
  \int_{\Omega} \bm{q}_T \phi_T^{(j)} \, dV, \qquad \text{for all } j.
\end{equation}
If the constant function \(1\) is contained in \(\mathcal{V}_T\),
conservation follows from the weak form and linearity of the integral
\begin{equation}
\int_{\Omega} \bm{q}_D \, dV =
  \int_{\Omega} \bm{q}_T \, dV.
\end{equation}
Expanding \(\bm{q}_D\) and \(\bm{q}_T\) with respect to their coefficients
\(\bm{d}\) and \(\bm{t}\), the weak form gives rise to a linear system
\begin{equation}\label{eq:weak-form-system}
M_T \bm{t} = M_{TD} \bm{d}.
\end{equation}
Here \(M_T\) is the mass matrix for
\(\mathcal{M}_T\) given by
\begin{equation}
\left(M_T\right)_{ij} = \int_{\Omega} \phi_T^{(i)} \phi_T^{(j)} \, dV.
\end{equation}
In the discontinuous Galerkin case, \(M_T\) is block diagonal with blocks
that correspond to each element, so \eqref{eq:weak-form-system} can be
solved locally on each element \(\mathcal{T}\) in the target mesh. By
construction, \(M_T\) is symmetric and sparse since \(\left(M_T\right)_{ij}\)
will be \(0\) unless \(\phi_T^{(i)}\) and \(\phi_T^{(j)}\) are supported
on the same element \(\mathcal{T}\). In the continuous case, \(M_T\) is
globally coupled since coefficients corresponding to boundary nodes interact
with multiple elements. The matrix \(M_{TD}\) is a ``mixed'' mass matrix
between the target and donor meshes:
\begin{equation}
\left(M_{TD}\right)_{ij} = \int_{\Omega} \phi_T^{(i)} \phi_D^{(j)} \, dV.
\end{equation}
Boundary conditions can be imposed on the system by fixing some coefficients,
but that is equivalent to removing some of the basis functions which may
\rev{in turn} make the projection non-conservative. This is because the removed
basis functions may have been used in \(1 = \sum_j u_j \phi_{T}^{(j)}\).

Computing \(M_T\) is fairly straightforward since the (bidirectional) mapping
from elements \(\mathcal{T}\) to basis functions \(\phi_T^{(j)}\) supported
on those elements is known. When using shape functions in the
global coordinates basis (see Section~\ref{subsec:shape-functions}), the
integrand \(F = \phi_T^{(i)} \phi_T^{(j)}\) will be a polynomial of degree
\(2p\) on \(\reals^2\). The domain of integration \(\mathcal{T}
= b\left(\utri\right)\) is the image of a (degree \(p\)) map \(b(s, t)\)
from the unit triangle. Using substitution
\begin{equation}\label{eq:mass-mat-subst}
\int_{b\left(\utri\right)} F(x, y) \, dx \, dy =
  \int_{\utri} \det(Db) F\left(x(s, t), y(s, t)\right) \, ds \, dt
\end{equation}
(we know the map preserves orientation, i.e. \(\det(Db)\) is positive).
Once transformed this way, a quadrature rule on the unit
triangle (\cite{Dunavant1985}) can be used.

On the other hand, computing \(M_{TD}\) is significantly more
challenging. This requires solving both a geometric problem ---
finding the region to integrate over --- and an analytic
problem --- computing the integrals. The integration can be done with
a quadrature rule, though finding this region is significantly
more difficult.

\subsection{Common Refinement}

\begin{figure}
  \includegraphics{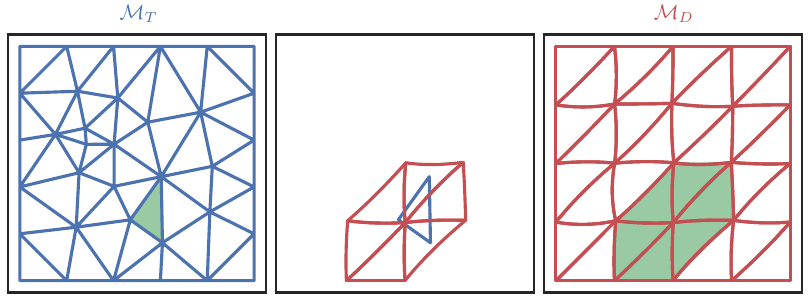}
  \centering
  \captionsetup{width=.75\linewidth}
  \caption{All donor elements that cover a target element}
  \label{fig:target-elt-all-matching}
\end{figure}

Rather than computing \(M_{TD}\), the right-hand side
of~\eqref{eq:weak-form-system} can be computed directly via
\begin{equation}\label{eq:lumped-mixed-mass-matrix}
\left(M_{TD} \bm{d}\right)_j = \int_{\Omega} \phi_T^{(j)} \bm{q}_D \, dV.
\end{equation}
Any given \(\phi\) is supported on an element \(\mathcal{T}\) in the
target mesh. Since \(\bm{q}_D\) is piecewise defined over each element
\(\mathcal{T}'\) in the donor mesh, the
integral~\eqref{eq:lumped-mixed-mass-matrix} may be problematic.
In the continuous Galerkin case, \(\bm{q}_D\) need not be differentiable
across \(\mathcal{T}\) and in the discontinuous Galerkin case,
\(\bm{q}_D\) need not even be continuous. This necessitates a
partitioning of the domain:
\begin{equation}
\int_{\Omega} \phi \, \bm{q}_D \, dV =
  \int_{\mathcal{T}} \phi \, \bm{q}_D \, dV =
  \sum_{\mathcal{T}' \in \mathcal{M}_D} \int_{\mathcal{T} \cap \mathcal{T}'}
    \phi \left.\bm{q}_D\right|_{\mathcal{T}'} \, dV.
\end{equation}
In other words, the integral over \(\mathcal{T}\) splits into integrals
over intersections \(\mathcal{T} \cap \mathcal{T}'\) for all
\(\mathcal{T}'\) in the donor mesh that intersect \(\mathcal{T}\)
(Figure~\ref{fig:target-elt-all-matching}). Since both \(\phi\) and
\(\left.\bm{q}_D\right|_{\mathcal{T}'}\) are polynomials on
\(\mathcal{T} \cap \mathcal{T}'\), the integrals will be exact when
using a quadrature scheme of an appropriate degree of accuracy.
Without partitioning \(\mathcal{T}\), the integrand is not a polynomial
(in fact, possibly not smooth), so the quadrature cannot be exact.

In order to compute \(M_{TD} \bm{d}\), we'll need to compute the
\emph{common refinement}, i.e. an intermediate mesh that contains
both the donor and target meshes. This will consist of all non-empty
\(\mathcal{T} \cap \mathcal{T}'\) as \(\mathcal{T}\) varies over
elements of the target mesh and \(\mathcal{T}'\) over elements of the
donor mesh. This requires solving three specific subproblems:
\begin{itemize}
%% H/T: https://tex.stackexchange.com/a/6086/32270
\itemsep 0em
\item Forming the region(s) of intersection between two elements that
  are B\'{e}zier triangles.
\item Finding all pairs of elements, one each from the target and donor mesh,
  that intersect.
\item Numerically integrating over a region of intersection between two
  elements.
\end{itemize}
The first subproblem is discussed in
Section~\ref{subsec:intersect-bez-tri}
and the curved polygon region(s) of intersection have been described
in Section~\ref{subsec:curved-polygons}. The second will be considered
in Section~\ref{subsec:expanding-front} and the third in
Section~\ref{subsec:integration-on-curved} below.

\subsection{Expanding Front}\label{subsec:expanding-front}

\begin{figure}
  \includegraphics{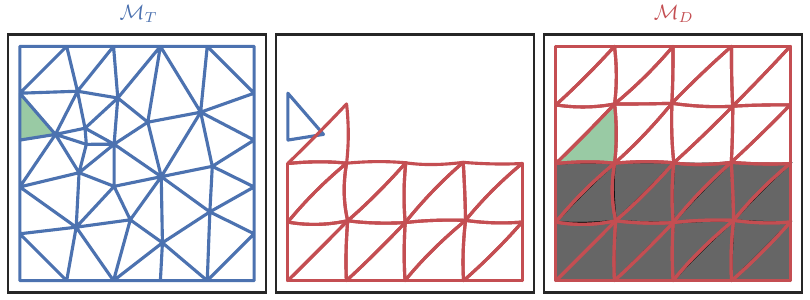}
  \centering
  \captionsetup{width=.75\linewidth}
  \caption{Brute force search for a donor element \(\mathcal{T}'\) that matches
    a fixed target element \(\mathcal{T}\).}
  \label{fig:target-elt-brute-force}
\end{figure}

We seek to identify all pairs \(\mathcal{T}\) and \(\mathcal{T}'\) of
intersecting target and donor elements. The na\"{i}ve approach just
considers every pair
of elements and takes \(\bigO{\left|\mathcal{M}_D\right|
\left|\mathcal{M}_T\right|}\) to complete.\footnote{For a mesh \(\mathcal{M}\),
the expression \(\left|\mathcal{M}\right|\) represents the number of elements
in the mesh.} Taking after \cite{Farrell2011}, we can do much better than this
quadratic time search. In fact, we can compute all integrals in
\(\bigO{\left|\mathcal{M}_D\right| + \left|\mathcal{M}_T\right|}\).
First, we fix an element of the target mesh and perform a brute-force search
to find an intersecting element in the donor mesh
(Figure~\ref{fig:target-elt-brute-force}). This has worst-case time
\(\bigO{\left|\mathcal{M}_D\right|}\).

\begin{figure}
  \includegraphics{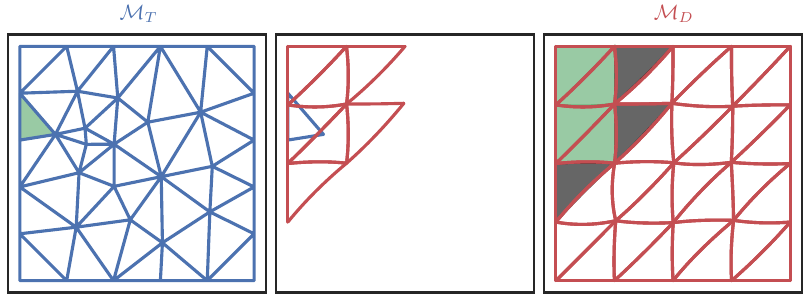}
  \centering
  \captionsetup{width=.75\linewidth}
  \caption{All donor elements \(\mathcal{T}'\) that cover a target element
    \(\mathcal{T}\), with an extra layer of neighbors in the donor mesh
    that \emph{do not} intersect \(\mathcal{T}\).}
  \label{fig:target-elt-overlap-extra-layer}
\end{figure}

Once we have such a match, we use the connectivity graph of the
donor mesh to perform a breadth-first search
for neighbors that also intersect the target
element \(\mathcal{T}\) (Figure~\ref{fig:target-elt-overlap-extra-layer}).
This search takes \(\bigO{1}\) time. It's also worthwhile to keep the first
layer of donor elements that don't intersect \(\mathcal{T}\) because they are
more likely to intersect the neighbors of \(\mathcal{T}\)\footnote{In the
very unlikely case that the boundary of \(\mathcal{T}\)
\emph{exactly matches} the boundaries of the donor elements that cover it,
\emph{none} of the overlapping donor elements can intersect the neighbors of
\(\mathcal{T}\) so the first layer of non-intersected donor
elements must be considered.}. Using the list of intersected elements, a
neighbor of \(\mathcal{T}\) can find a donor element it intersects with in
\(\bigO{1}\) time (Figure~\ref{fig:target-elt-neighbor}).
As seen, after our \(\bigO{\left|\mathcal{M}_D\right|}\)
initial brute-force search, the localized intersections for each
target element \(\mathcal{T}\) take
\(\bigO{1}\) time. So together, the process takes
\(\bigO{\left|\mathcal{M}_D\right| + \left|\mathcal{M}_T\right|}\).

For special cases, e.g. ALE methods, the initial brute-force search
can be reduced to \(\bigO{1}\). This could be enabled by tracking
the remeshing process so a correspondence already exists. If a full
mapping from donor to target mesh exists, the process of computing
\(M_T\) and \(M_{TD} \bm{d}\) can be fully parallelized across
elements of the target mesh, or even across shape functions
\(\phi_T^{(j)}\). \rev{Spatial binning techniques for mesh elements could also be
used to improve the initial brute-force search as in \cite{Weiss16_imr}.}

\begin{figure}
  \includegraphics{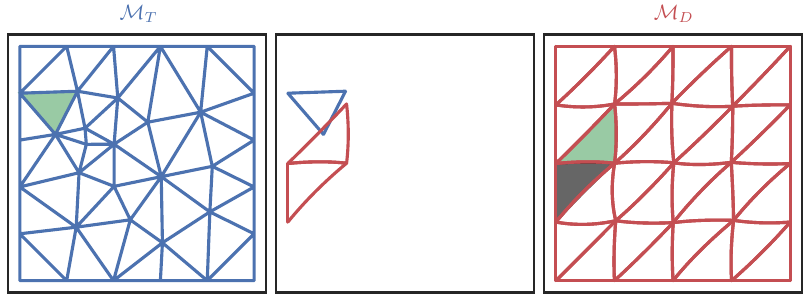}
  \centering
  \captionsetup{width=.75\linewidth}
  \caption{First match between a neighbor of the previously considered
    target element and all the donor elements that match the
    previously considered target element.}
  \label{fig:target-elt-neighbor}
\end{figure}

\subsection{Integration over Curved Polygons}
\label{subsec:integration-on-curved}

In order to numerically evaluate integrals of the form
\begin{equation}
\int_{\mathcal{T}_0 \cap \mathcal{T}_1} F(x, y) \, dV
\end{equation}
we must have a quadrature rule on these curved polygon
(Section~\ref{subsec:curved-polygons}) intersections
\(\mathcal{T}_0 \cap \mathcal{T}_1\).
To do this, we transform the integral into several line integrals
via Green's theorem and then use an exact Gaussian quadrature to
evaluate them. Throughout this section we'll assume the integrand
is of the form \(F = \phi_0 \phi_1\) where each \(\phi_j\) is a shape
function on \(\mathcal{T}_j\). In addition, we'll refer to the two
B\'{e}zier maps that define the elements being intersected:
\(\mathcal{T}_0 = b_0\left(\utri\right)\) and
\(\mathcal{T}_1 = b_1\left(\utri\right)\).

First a discussion of a method not used. A somewhat simple approach would be
be to use polygonal approximation. I.e. approximate the boundary of each
B\'{e}zier triangle with line segments, intersect the resulting polygons,
triangulate the intersection polygon(s) and then numerically integrate
on each triangulated cell. However, this approach is prohibitively inefficient.
For example, consider computing the area of \(\mathcal{P}\) via an
integral: \(\int_{\mathcal{P}} 1 \, dV\). By using the actual curved
boundaries of each element, this integral can be computed with relative
error on the order of machine precision \(\bigO{\mach}\). On the other hand,
approximating each side of \(\mathcal{P}\) with \(N\) line segments, the
relative error is \(\bigO{1/N^2}\). (For example, if \(N = 2\) an edge
curve \(b(s, 0)\) would be replaced by segments connecting \(b(0, 0),
b(1/2, 0)\) and \(b(1, 0)\), i.e. \(N + 1\) equally spaced parameters.) This
means that in order to perform as well as an \emph{exact} quadrature used in
the curved case, we'd need \(N = \bigO{1/\sqrt{\mach}}\).

Since polygonal approximation is prohibitively expensive,
we work directly with curved edges and compute integrals on a regular
domain via substitution.
If two B\'{e}zier triangles intersect with positive measure, then
the region of intersection is one or more disjoint curved polygons:
\(\mathcal{T}_0 \cap \mathcal{T}_1 = \mathcal{P}\) or
\(\mathcal{T}_0 \cap \mathcal{T}_1 = \mathcal{P} \cup
\mathcal{P}' \cup \cdots\).
The second case can be handled in the same way as the first by handling each
disjoint region independently:
\begin{equation}
\int_{\mathcal{T}_0 \cap \mathcal{T}_1} F(x, y) \, dV =
  \int_{\mathcal{P}} F(x, y) \, dV +
  \int_{\mathcal{P}'} F(x, y) \, dV + \cdots.
\end{equation}
Each curved polygon \(\mathcal{P}\) is defined by its boundary, a
piecewise smooth parametric curve:
\begin{equation}
\partial \mathcal{P} = C_1 \cup \cdots \cup C_n.
\end{equation}
This can be thought of as a polygon with \(n\) sides that happens to
have curved edges.

Quadrature rules for straight sided polygons have been
studied (\cite{Mousavi2009}) though they are not in wide use. Even
if a polygonal quadrature rule was to be employed, a map would need to be
established from a reference polygon onto the curved edges. This map
could then be used with a change of coordinates to move the integral
from the curved polygon to the reference polygon. The problem of extending
a mapping from a boundary to an entire domain has been studied as
transfinite interpolation
(\cite{chenin:tel-00284680, Gordon1982, Perronnet1998}),
barycentric coordinates (\cite{Wachspress1975}) or mean value coordinates
(\cite{Floater2003}). However, these maps aren't typically suitable for
numerical integration because they are either not bijective or increase the
degree of the edges.

\begin{figure}
  \includegraphics{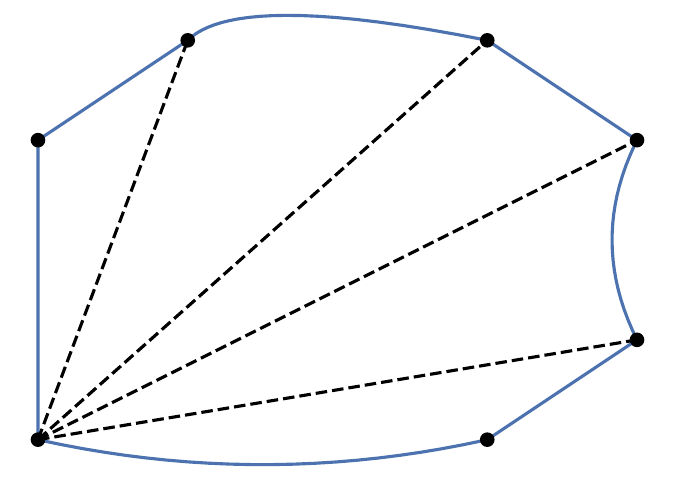}
  \centering
  \captionsetup{width=.75\linewidth}
  \caption{Curved polygon tessellation, done by introducing diagonals
    from a single vertex node.}
  \label{fig:tessellated-curved-polygon}
\end{figure}

Since simple and well established quadrature rules do exist for triangles,
a valid approach would be
to tessellate a curved polygon into valid B\'{e}zier triangles
(Figure~\ref{fig:tessellated-curved-polygon}) and then use substitution as
in~\eqref{eq:mass-mat-subst}. However, tessellation is challenging for both
theoretical and computational reasons.
Theoretical: it's not even clear if an arbitrary curved polygon \emph{can}
be tessellated into B\'{e}zier triangles. Computational: the placement
of diagonals and potential introduction of interior nodes is very
complex in cases where the curved polygon is nonconvex. What's more,
the curved polygon is given only by the boundary, so higher degree triangles
(i.e. cubic and above) introduced during tessellation would need to place
interior control points without causing the triangle to invert. To get a
sense for these challenges, note how the ``simple'' introduction of
diagonals in Figure~\ref{fig:bad-tessellation-required} leads to one
inverted element (the gray element) and another element with area outside
of the curved polygon (the yellow element). Inverted B\'{e}zier triangles
are problematic because the accompanying mapping leaves the boundary
established by the edge curves. For example,
if the tessellation of \(\mathcal{P}\) contains an inverted
B\'{e}zier triangle \(\mathcal{T}_2 = b\left(\utri\right)\) then we'll
need to numerically integrate
\begin{equation}
\int_{\mathcal{T}_2} \phi_0 \phi_1 \, dx \, dy =
  \int_{\utri} \left|\det(Db)\right| \left(\phi_0 \circ b\right)
  \left(\phi_1 \circ b\right) \, ds \, dt.
\end{equation}
If the shape functions are from the pre-image basis (see
Section~\ref{subsec:shape-functions}), then \(\phi_0\)
will not be defined at \(b(s, t) \not\in \mathcal{T}_0\) (similarly for
\(\phi_1\)). Additionally, the absolute value in \(\left|\det(Db)\right|\)
makes the integrand non-smooth since for inverted elements
\(\det(Db)\) takes both signs. If the shape functions are from the
global coordinates basis, then tessellation can be used via an
application of Theorem~\ref{theorem:bad-triangle}, however this involves
more computation than just applying Green's theorem along
\(\partial \mathcal{P}\). By applying the theorem, inverted elements may
be used in a tessellation, for example by introducing artificial
diagonals from a vertex as in Figure~\ref{fig:tessellated-curved-polygon}.
In the global coordinates basis, \(F = \phi_0 \phi_1\) \emph{can}
be evaluated for points in an inverted element that leave
\(\mathcal{T}_0\) or \(\mathcal{T}_1\).

\begin{figure}
  \includegraphics{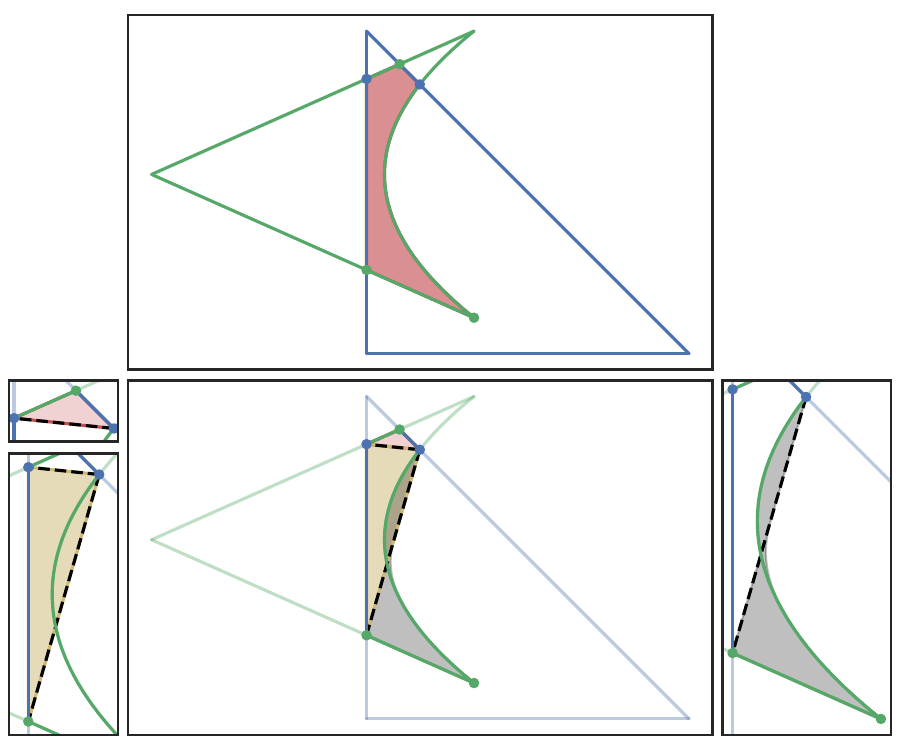}
  \centering
  \captionsetup{width=.75\linewidth}
  \caption{Curved polygon intersection that can't be tessellated
    (with valid B\'{e}zier triangles) by introducing diagonals.}
  \label{fig:bad-tessellation-required}
\end{figure}

Instead, we focus on a Green's theorem based approach.
Define horizontal and vertical antiderivatives
\(H, V\) of the integrand \(F\) such that \(H_x = V_y = F\).
We make these \emph{unique} by imposing the
extra condition that \(H(0, y) \equiv V(x, 0) \equiv 0\).
This distinction is arbitrary, but in order to \emph{evaluate}
\(H\) and \(V\), the univariate functions \(H(0, y)\) and \(V(x, 0)\)
must be specified.
Green's theorem tells us that
\begin{equation}
\int_{\mathcal{P}} 2 F \, dV =
\int_{\mathcal{P}} H_x + V_y \, dV =
\oint_{\partial \mathcal{P}} H \, dy - V \, dx =
\sum_j \int_{C_j} H \, dy - V \, dx.
\end{equation}
For a given curve \(C\) with components \(x(r), y(r)\)
defined on the unit interval, this amounts to having
to integrate
\begin{equation}
G(r) = H(x(r), y(r)) y'(r) - V(x(r), y(r)) x'(r).
\end{equation}
To do this, we'll use Gaussian
quadrature with degree of accuracy sufficient to cover the degree
of \(G(r)\).

If the shape functions are in the global coordinates basis
(Section~\ref{subsec:shape-functions}), then \(G\) will also be a polynomial.
This is because for these shape functions \(F = \phi_0 \phi_1\) will be
polynomial on \(\reals^2\) and so will \(H\) and \(V\) and since each curve
\(C\) is a B\'{e}zier curve segment, the components are also polynomials.
If the shape functions are in the pre-image basis, then \(F\) won't in general
be polynomial, hence standard quadrature rules can't be exact.

To evaluate \(G\), we must also evaluate \(H\) and \(V\) numerically.
For example, since \(H(0, y) \equiv 0\),
the fundamental theorem of calculus tells us that
\(H\left(\alpha, \beta\right) = \int_0^{\alpha} F\left(x, \beta\right) \, dx\).
To compute this integral via Gaussian quadrature
\begin{equation}
H\left(\alpha, \beta\right) \approx \frac{\alpha}{2} \sum_j w_j
  F\left(\frac{\alpha}{2} (x_j + 1), \beta\right)
\end{equation}
we must be able to evaluate \(F\) for
points on the line \(y = \beta\) for \(x\) between \(0\) and \(\alpha\). If
the shape functions are from the pre-image basis, it may not even be possible
to evaluate \(F\) for such points. Since \(\left[\begin{array}{c c} \alpha &
\beta \end{array}\right]^T\) is on the boundary of \(\mathcal{P}\), without
loss of generality
assume it is on the boundary of \(\mathcal{T}_0\). Thus, for some elements
(e.g. if the point is on the bottom of the element), points
\(\left[\begin{array}{c c} \nu & \beta \end{array}\right]^T\) may not be
in \(\mathcal{T}_0\). Since \(\operatorname{supp}(\phi_0) = \mathcal{T}_0\),
we could take \(\phi_0(\nu, \beta) = 0\), but this would make the integrand
non-smooth and so the accuracy of the \emph{exact} quadrature would be
lost. But extending \(\phi_0 = \widehat{\phi}_0 \circ b_0^{-1}\) outside
of \(\mathcal{T}_0\) may be impossible: even though \(b_0\) is bijective on
\(\utri\) it may be many-to-one elsewhere hence \(b_0^{-1}\) can't be
reliably extended outside of \(\mathcal{T}_0\).

Even if the shape functions are from the global coordinates basis,
setting \(H(0, y) \equiv 0\) may introduce quadrature points that
are very far from \(\mathcal{P}\). This can be somewhat addressed by using
\(H(m, y) \equiv 0\) for a suitably chosen \(m\) (e.g. the minimum
\(x\)-value in \(\mathcal{P}\)). Then we have
\(H\left(\alpha, \beta\right) = \int_m^{\alpha} F\left(x, \beta\right) \, dx\).

\section{Numerical Experiments}

\begin{figure}
  \includegraphics{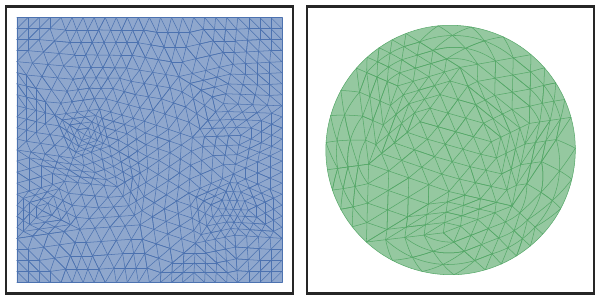}
  \centering
  \captionsetup{width=.75\linewidth}
  \caption{Some example meshes used during the numerical
    experiments; the donor meshs is on the left / blue and the target
    mesh is on the right / green. These meshes are the cubic approximations
    of each domain and have been refined twice.}
  \label{fig:meshes-used-refined}
\end{figure}

A numerical experiment was performed to investigate the observed order
of convergence. Three pairs of random meshes were generated, the donor
on a square of width \(17 / 8\) centered at the origin
and the target on the unit disc. These domains were chosen intentionally so
that the target mesh was completely covered by the donor mesh and the
boundaries did not accidentally introduce ill-conditioned intersection between
elements. The pairs were linear, quadratic and
cubic approximations of the domains. The convergence test was done by
refining each pair of meshes four times and performing solution transfer
at each level. Figure~\ref{fig:meshes-used-refined} shows the cubic
pair of meshes after two refinements.

\begin{figure}
  \includegraphics{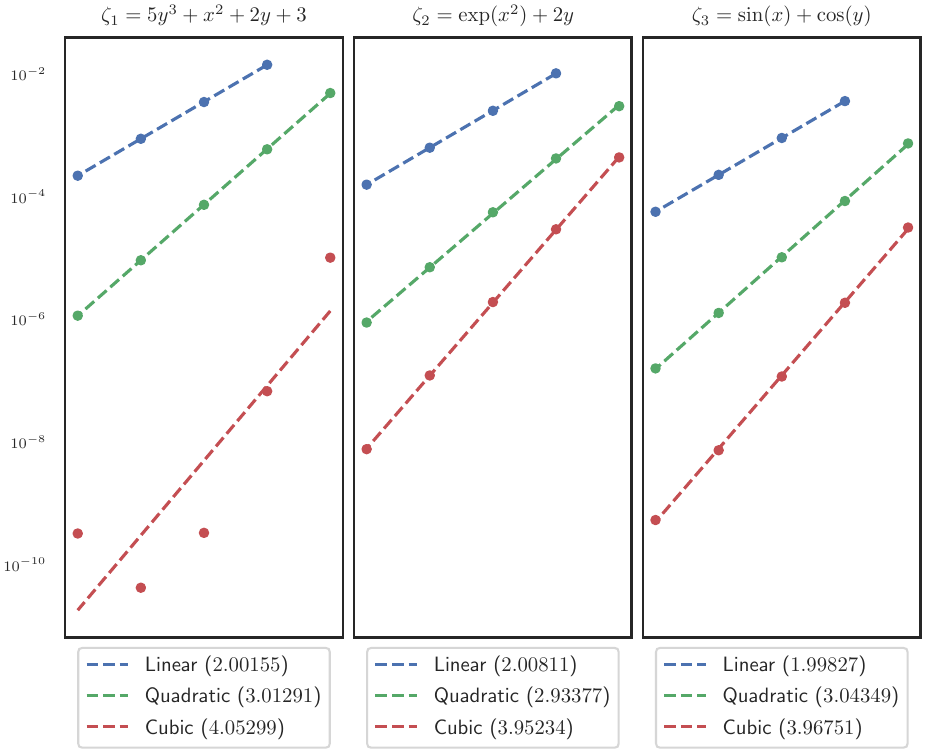}
  \centering
  \captionsetup{width=.75\linewidth}
  \caption{Convergence results for scalar fields on three pairs of related
    meshes: a linear, quadratic and cubic mesh of the same domain.}
  \label{fig:composite-errors}
\end{figure}

Taking after \cite{Farrell2011}, we transfer three
discrete fields in the discontinuous Galerkin (DG) basis.
Each field is derived from one of the smooth scalar functions
\begin{align}
\zeta_1(x, y) &= 5 y^3 + x^2 + 2y + 3 \\
\zeta_2(x, y) &= \exp\left(x^2\right) + 2y \\
\zeta_3(x, y) &= \sin(x) + \cos(y).
\end{align}
For a given mesh size \(h\), we expect that on a degree \(p\) isoparametric
mesh our solution transfer will have \(\bigO{h^{p + 1}}\) errors. To measure
the rate of convergence:
\begin{itemize}
\itemsep 0em
\item Choose a mesh pair \(\mathcal{M}_D\), \(\mathcal{M}_T\) and
a known function \(\zeta(x, y)\).
\item Refine the meshes recursively, starting with
  \(\mathcal{M}_D^{(0)} = \mathcal{M}_D\),
  \(\mathcal{M}_T^{(0)} = \mathcal{M}_T\).
\item Create meshes
\(\mathcal{M}_D^{(j)}\) and \(\mathcal{M}_T^{(j)}\) from
\(\mathcal{M}_D^{(j - 1)}\)
and \(\mathcal{M}_T^{(j - 1)}\) by subdividing each curved
element into four elements.
\item Approximate \(\zeta\) by a discrete field: the nodal
interpolant on the donor mesh \(\mathcal{M}_D^{(j)}\). The
nodal interpolant is constructed by evaluating \(\zeta\) at the
nodes \(\bm{n}_j\) corresponding to each shape function:
\begin{equation}
\bm{f}_j = \sum_i \zeta\left(\bm{n}_i\right) \phi_D^{(i)}.
\end{equation}
\item Transfer \(\bm{f}_j\) to the discrete field \(\bm{g}_j\) on
the target mesh  \(\mathcal{M}_T^{(j)}\).
\item Compute the relative error on \(\mathcal{M}_T^{(j)}\):
\(E_j = \| \bm{g}_j - \zeta \|_2 / \| \zeta \|_2\) (here \(\| \cdot \|_2\)
is the \(L_2\) norm on the target mesh).

We should instead be measuring \(\| \bm{g}_j - \bm{f}_j \|_2 /
\| \bm{f}_j \|_2\), but \(E_j\) is much easier to
compute for \(\zeta(x, y)\) that are straightforward to evaluate.
Due to the triangle inequality \(E_j\) can be a reliable proxy
for the actual projection error, though when
\(\| \bm{f}_j - \zeta \|_2\) becomes too large it will dominate
the error and no convergence will be observed.
\end{itemize}

Convergence results are shown in Figure~\ref{fig:composite-errors}
and confirm the expected orders. Since \(\zeta_1\) is a cubic polynomial,
one might expect the solution transfer on the cubic mesh to be \emph{exact}.
This expectation would hold in the superparametric case (i.e. where the
elements are straight sided but the shape functions are higher degree).
However, when the elements aren't straight \rev{sided} the isoparametric mapping
changes the underlying function space.

\section{Conclusion}

This paper has described a method for conservative interpolation
between curved meshes. The transfer process conserves globally
to machine precision since we can use exact quadratures for all
integrals. The primary source of error comes from solving the linear
system with the mass matrix for the target mesh. This allows
less restrictive usage of mesh adaptivity, which can make computations
more efficient. Additionally, having a global transfer algorithm
allows for remeshing to be done less frequently.

The algorithm breaks down into three core subproblems: B\'{e}zier triangle
intersection, a expanding front for intersecting elements and
integration on curved polygons. The inherently local nature of the
expanding front allows the algorithm to be parallelized via domain
decomposition with little data shared between processes. By
restricting integration to the intersection of elements from the
target and donor meshes, the algorithm can accurately transfer
both continuous and discontinuous fields.

As mentioned in the preceding chapters, there are several research
directions possible to build upon the solution transfer algorithm.
The usage of Green's theorem nicely extends to \(\reals^3\) via
Stokes' theorem, but the B\'{e}zier triangle intersection algorithm
is specific to \(\reals^2\). The equivalent B\'{e}zier tetrahedron
intersection algorithm is significantly more challenging.

The restriction to shape functions from the global coordinates basis
is a symptom of the method and not of the inherent problem. The
pre-image basis has several appealing properties, for example
this basis can be precomputed on \(\utri\). The problem of a
valid tessellation of a curved polygon warrants more exploration.
Such a tessellation algorithm would enable usage of the pre-image
basis.

The usage of the global coordinates basis does have some benefits.
In particular, the product of shape functions from different meshes
is still a polynomial in \(\reals^2\). This means that we could
compute the coefficients of \(F = \phi_0 \phi_1\) directly and
use them to evaluate the antiderivatives \(H\) and \(V\) rather
than using the fundamental theorem of calculus. Even if this
did not save any computation, it may still be preferred over
the FTC approach because it would remove the usage of quadrature
points outside of the domain \(\mathcal{P}\).

\bibliography{paper}
\bibliographystyle{alpha}

\appendix

\section{B\'{e}zier Intersection Problems}\label{sec:bezier-intersection}

\subsection{Intersecting B\'{e}zier Curves}

The problem of intersecting two B\'{e}zier curves is a core building
block for intersecting two B\'{e}zier triangles in \(\reals^2\).
Since a curve is an algebraic variety of dimension one,
the intersections will either be a curve segment common to both curves (if
they coincide) or a finite set of points (i.e. dimension zero).
Many algorithms have been described in the literature, both
geometric (\cite{Sederberg1986, Sederberg1990, Kim1998}) and
algebraic (\cite{Manocha:CSD-92-698}).

In the implementation for this paper, the B\'{e}zier subdivision
algorithm is used.
In the case of a transversal intersection (i.e. one where the
tangents to each curve are not parallel and both are non-zero),
this algorithm performs very well. However, when curves are tangent,
a large number of (false) candidate intersections are detected and
convergence of Newton's method slows once in a neighborhood of an
actual intersection. Non-transversal intersections
have infinite condition number, but transversal intersections with
very high condition number can also cause convergence problems.

\begin{figure}
  \includegraphics{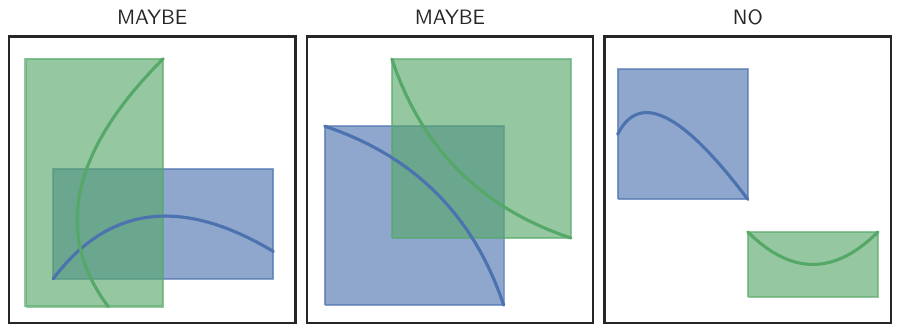}
  \centering
  \captionsetup{width=.75\linewidth}
  \caption{Bounding box intersection predicate. This is a cheap way to
    conclude that two curves don't intersect, though it inherently is
    susceptible to false positives.}
  \label{fig:bounding-box-check}
\end{figure}

In the B\'{e}zier subdivision algorithm, we first check if the
bounding boxes for the curves are disjoint
(Figure~\ref{fig:bounding-box-check}).
We use the bounding boxes
rather than the convex hulls since they are easier to compute and
the intersections of boxes are easier to check.
If they are disjoint, the pair can be rejected. If not, each curve
\(\mathcal{C} = b\left(\left[0, 1\right]\right)\) is split into two halves
by splitting the unit interval: \(b\left(\left[0, \frac{1}{2}\right]\right)\)
and \(b\left(\left[\frac{1}{2}, 1\right]\right)\)
(Figure~\ref{fig:bezier-curve-subdivision}).

\begin{figure}
  \includegraphics{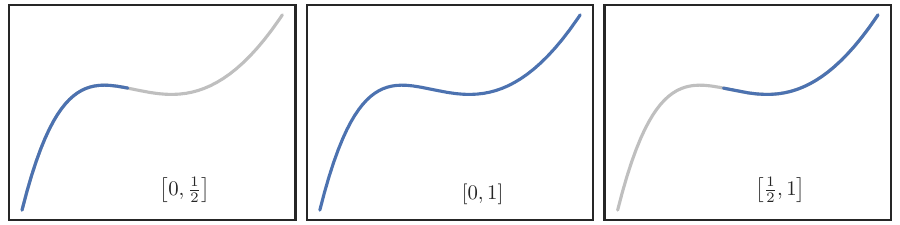}
  \centering
  \captionsetup{width=.75\linewidth}
  \caption{B\'{e}zier curve subdivision.}
  \label{fig:bezier-curve-subdivision}
\end{figure}

\begin{figure}
  \includegraphics{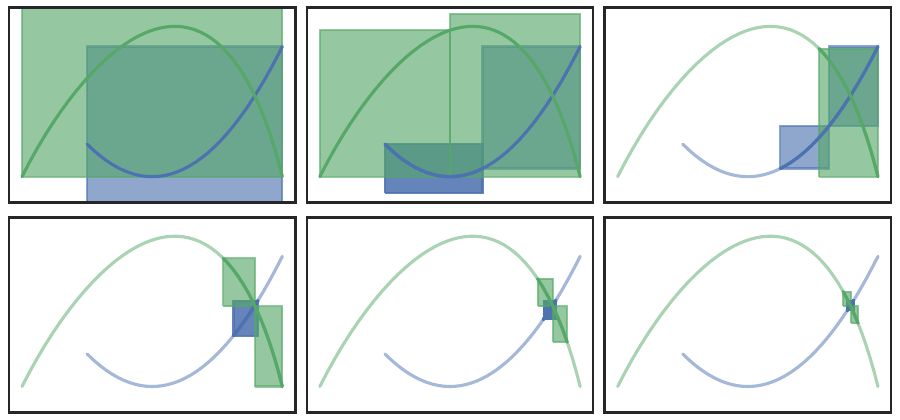}
  \centering
  \captionsetup{width=.75\linewidth}
  \caption{B\'{e}zier subdivision algorithm.}
  \label{fig:bezier-subdivision-process}
\end{figure}

As the subdivision continues,
some pairs of curve segments may be kept around that won't lead to an
intersection (Figure~\ref{fig:bezier-subdivision-process}).
Once the curve segments are close to linear within a given tolerance
(Figure~\ref{fig:bezier-subdivision-linearized}), the process
terminates.

\begin{figure}
  \includegraphics{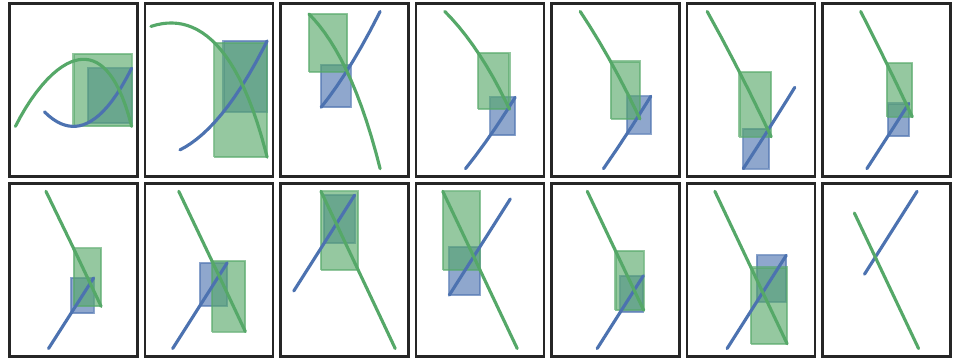}
  \centering
  \captionsetup{width=.75\linewidth}
  \caption{Subdividing until linear within tolerance.}
  \label{fig:bezier-subdivision-linearized}
\end{figure}

Once both curve segments are linear (to tolerance), the intersection is
approximated by intersecting the lines connecting the endpoints of each
curve segment. This approximation is used as a starting point for Newton's
method, to find a root of \(F(s, t) = b_0(s) - b_1(t)\). Since
\(b_0(s), b_1(t) \in \reals^2\) we have Jacobian \(J =
\left[ \begin{array}{c c} b_0'(s) & -b_1'(t) \end{array}\right]\).
With these, Newton's method is
\begin{equation}
\left[ \begin{array}{c c} s_{n + 1} & t_{n + 1} \end{array}\right]^T =
\left[ \begin{array}{c c} s_n & t_n \end{array}\right]^T -
J_n^{-1} F_n.
\end{equation}
This also gives an indication why convergence issues occur at non-transveral
intersections: they are exactly the intersections where the Jacobian is
singular.

\subsection{Intersecting B\'{e}zier Triangles}\label{subsec:intersect-bez-tri}

The chief difficulty in intersecting two surfaces is intersecting their edges,
which are B\'{e}zier curves.
Though this is just a part of the overall algorithm, it proved to be the
\emph{most difficult} to implement (\cite{Hermes2017}). So the first part
of the algorithm is to find all points
where the edges intersect (Figure~\ref{fig:edge-intersections}).

\begin{figure}
  \includegraphics{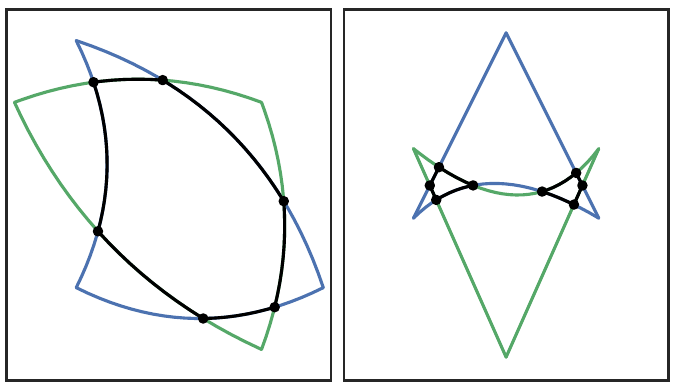}
  \centering
  \captionsetup{width=.75\linewidth}
  \caption{Edge intersections during B\'{e}zier triangle intersection.}
  \label{fig:edge-intersections}
\end{figure}

To determine the curve segments that bound the curved polygon region(s)
(see Section~\ref{subsec:curved-polygons} for more about curved polygons) of
intersection, we not only need to keep track
of the coordinates of intersection, we also need to keep note of
\emph{which} edges the intersection occurred on and the parameters along
each curve.
With this information, we can classify each point of intersection
according to which of the two curves forms the boundary of the
curved polygon (Figure~\ref{fig:intersection-classification}).
Using the right-hand rule we can compare the tangent
vectors on each curve to determine which one is on the interior.

\begin{figure}
  \includegraphics{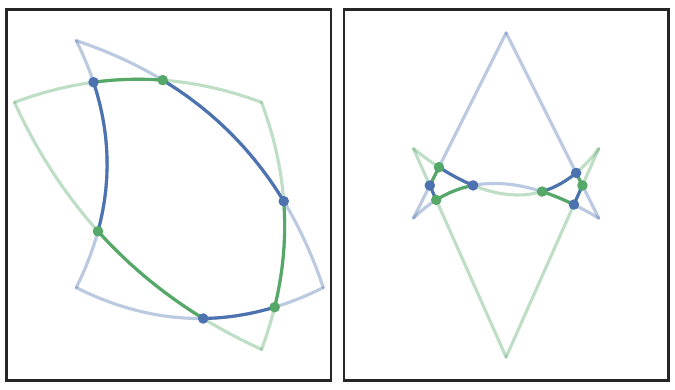}
  \centering
  \captionsetup{width=.75\linewidth}
  \caption{Classified intersections during B\'{e}zier triangle intersection.}
  \label{fig:intersection-classification}
\end{figure}

This classification becomes more difficult when the curves
are tangent at an intersection, when the intersection occurs at a corner
of one of the surfaces or when two intersecting edges are coincident
on the same algebraic curve (Figure~\ref{fig:intersection-difficulties}).

\begin{figure}
  \includegraphics{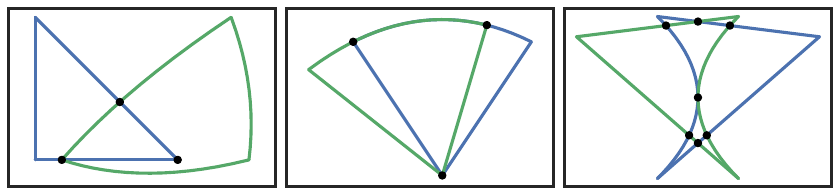}
  \centering
  \captionsetup{width=.75\linewidth}
  \caption{B\'{e}zier triangle intersection difficulties.}
  \label{fig:intersection-difficulties}
\end{figure}

In the case of tangency, the intersection is non-transversal, hence has
infinite condition number. In the case of coincident curves, there are
infinitely many intersections (along the segment when the curves
coincide) so the subdivision process breaks down.

\rev{In all of our experiments, we did not encounter any of these difficulties with our implementation. However, for fail-safe code, additional fallback strategies would be required, and we refer the reader to the extensive literature for further details on covering all possible special cases that might occur \cite{Sederberg1986,Sederberg1990,Kim1998}.}

\subsubsection{Example}

Consider two B\'{e}zier surfaces
(Figure~\ref{fig:surface-surface-example})

\begin{equation}
b_0(s, t) =
\left[ \begin{array}{c}
    8 s \\ 8 t \end{array}\right] \qquad
b_1(s, t) =
\left[ \begin{array}{c}
    2 (6 s + t - 1) \\
    2 (8 s^2 + 8 s t - 8 s + 3 t + 2) \end{array}\right]
\end{equation}
\begin{figure}
  \includegraphics{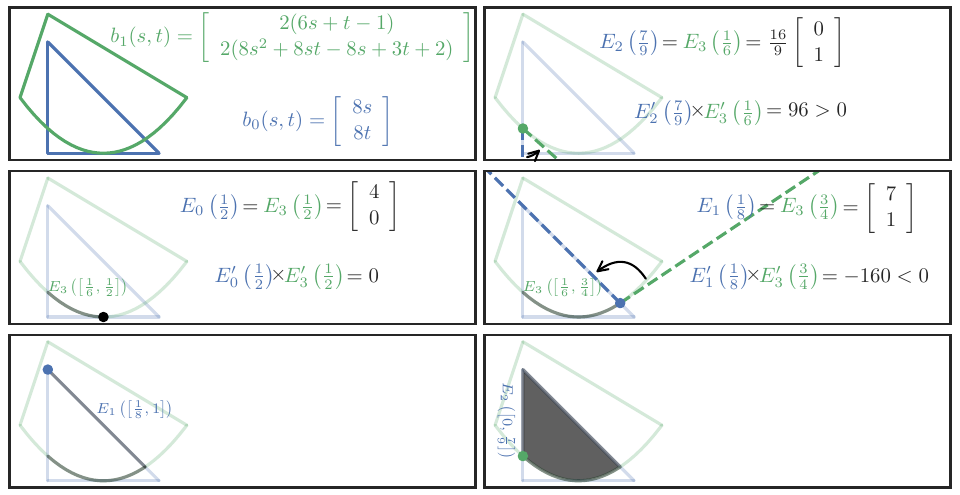}
  \centering
  \captionsetup{width=.75\linewidth}
  \caption{Surface Intersection Example}
  \label{fig:surface-surface-example}
\end{figure}

In the \emph{first step} we find all intersections of the
edge curves
\begin{multline}
E_0(r) = \left[ \begin{array}{c} 8 r \\ 0 \end{array}\right],
E_1(r) = \left[ \begin{array}{c} 8 (1 - r) \\ 8 r \end{array}\right],
E_2(r) = \left[ \begin{array}{c} 0 \\ 8 (1 - r) \end{array}\right], \\
E_3(r) = \left[ \begin{array}{c} 2 (6 r - 1) \\ 4 (2 r - 1)^2
  \end{array}\right],
E_4(r) = \left[ \begin{array}{c} 10 (1 - r) \\ 2 (3 r + 2) \end{array}\right],
E_5(r) = \left[ \begin{array}{c} - 2 r \\ 2 (5 - 3 r) \end{array}\right].
\end{multline}
We find three intersections
and we classify each of them by comparing the tangent vectors
\begin{gather}
  I_1:
  E_2\left(\frac{7}{9}\right) =
E_3\left(\frac{1}{6}\right) = \frac{16}{9}
\left[ \begin{array}{c} 0 \\ 1 \end{array}\right] \Longrightarrow
E_2'\left(\frac{7}{9}\right) \times
E_3'\left(\frac{1}{6}\right) = 96 \\
I_2:
E_0\left(\frac{1}{2}\right) =
E_3\left(\frac{1}{2}\right) =
\left[ \begin{array}{c} 4 \\ 0 \end{array}\right] \Longrightarrow
E_0'\left(\frac{1}{2}\right) \times
E_3'\left(\frac{1}{2}\right) = 0 \\
  I_3:
E_1\left(\frac{1}{8}\right) =
E_3\left(\frac{3}{4}\right) =
\left[ \begin{array}{c} 7 \\ 1 \end{array}\right] \Longrightarrow
E_1'\left(\frac{1}{8}\right) \times
E_3'\left(\frac{3}{4}\right) = -160.
\end{gather}
From here, we construct our curved polygon intersection by drawing
from our list of intersections until none remain.

\begin{itemize}
\itemsep 0em
\item First consider \(I_1\). Since
  \(E_2' \times
  E_3' > 0\)
  at this point, then we consider the curve
  \(E_3\) to be
  \emph{interior}.
\item After classification, we move along
  \(E_3\) until we
  encounter another intersection: \(I_2\)
\item \(I_2\) is a point of tangency since
  \(E_0'\left(\frac{1}{2}\right) \times
  E_3'\left(\frac{1}{2}\right) = 0\).
  Since a tangency has no impact on
  the underlying intersection geometry, we ignore it and
  keep moving.
\item Continuing to move along
  \(E_3\), we
  encounter another intersection: \(I_3\).
  Since
  \(E_1' \times
  E_3' < 0\)
  at this point, we consider the curve
  \(E_1\) to be
  \emph{interior} at the intersection. Thus we stop moving
  along \(E_3\)
  and we have our first curved segment:
  \(E_3\left(\left[
    \frac{1}{6}, \frac{3}{4}\right]\right)\)
\item Finding no other intersections on \(E_1\)
  we continue until the end of the edge.
  Now our (ordered) curved segments are:
  \begin{equation}
  E_3\left(\left[
    \frac{1}{6}, \frac{3}{4}\right]\right) \longrightarrow
  E_1\left(\left[
    \frac{1}{8}, 1\right]\right).
  \end{equation}
\item Next we stay at the corner and switch to the next curve
  \(E_2\), moving along that curve
  until we hit the next intersecton \(I_1\).
  Now our (ordered) curved segments are:
  \begin{equation}
  E_3\left(\left[
    \frac{1}{6}, \frac{3}{4}\right]\right) \longrightarrow
  E_1\left(\left[
    \frac{1}{8}, 1\right]\right) \longrightarrow
  E_2\left(\left[
    0, \frac{7}{9}\right]\right).
  \end{equation}
  Since we are now back where we started (at \(I_1\))
  the process stops
\end{itemize}
We represent the boundary of the curved polygon as B\'{e}zier curves, so
to complete the process we reparameterize (\cite[Ch.~5.4]{Farin2001}) each
curve onto the relevant interval. For example,
\(E_3\) has control points
\(p_0 = \left[ \begin{array}{c} -2 \\ 4 \end{array}\right]\),
\(p_1 = \left[ \begin{array}{c} 4 \\ -4 \end{array}\right]\),
\(p_2 = \left[ \begin{array}{c} 10 \\ 4 \end{array}\right]\)
and we reparameterize on \(\alpha = \frac{1}{6}, \beta = \frac{3}{4}\) to
control points
\begin{align}
  q_0 &= E_3\left(\frac{1}{6}\right) =
  \frac{16}{9} \left[ \begin{array}{c} 0 \\ 1 \end{array}\right] \\
  q_1 &= (1 - \alpha) \left[(1 - \beta) p_0 + \beta p_1\right] +
   \alpha \left[(1 - \beta) p_1 + \beta p_2\right] = \frac{1}{6} \left[
    \begin{array}{c} 21 \\ -8 \end{array}\right] \\
  q_2 &= E_3\left(\frac{3}{4}\right) = \left[
    \begin{array}{c} 7 \\ 1 \end{array}\right].
\end{align}

\subsection{B\'{e}zier Triangle Inverse}

The problem of determining the parameters \((s, t)\) given a point
\(\bm{p} = \left[\begin{array}{c c} x & y\end{array}\right]^T\)
in a B\'{e}zier triangle can also be solved by using
subdivision with a bounding box predicate and then Newton's method
at the end.

\begin{figure}
  \includegraphics{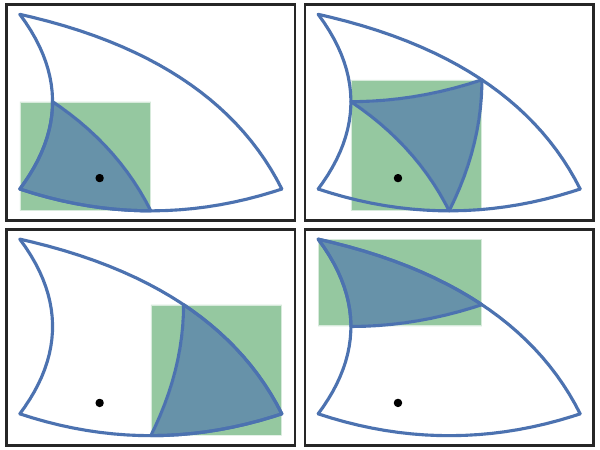}
  \centering
  \captionsetup{width=.75\linewidth}
  \caption{Checking for a point \(\bm{p}\) in each of four subregions
    when subdividing a B\'{e}zier triangle.}
  \label{fig:locate-in-triangle}
\end{figure}

For example, Figure~\ref{fig:locate-in-triangle} shows
how regions of \(\utri\) can be discarded recursively until the
suitable region for \((s, t)\) has a sufficiently small area. At
this point, we can apply Newton's method to the map \(F(s, t) =
b(s, t) - \bm{p}\). It's very helpful (for Newton's method) that
\(F: \reals^2 \longrightarrow \reals^2\) since the Jacobian will
always be invertible
when the B\'{e}zier triangle is valid. If \(\bm{p} \in \reals^3\)
then the system would be underdetermined. Similarly, if
\(\bm{p} \in \reals^2\) but \(b(s)\) is a B\'{e}zier curve then the
system would be overdetermined.

\section{Allowing Tessellation with Inverted Triangles}

\begin{theorem}\label{theorem:bad-triangle}
Consider three smooth curves
\(b_0, b_1, b_2\) that form a closed loop: \(b_0(1) = b_1(0)\),
\(b_1(1) = b_2(0)\) and \(b_2(1) = b_0(0)\).
Take \emph{any} smooth map \(\varphi(s, t)\) on \(\utri\) that
sends the edges to the three curves:
\begin{equation}
\varphi(r, 0) = b_0(r), \quad \varphi(1 - r, r) = b_1(r),
  \quad \varphi(0, 1 - r) = b_2(r) \quad \text{for } r \in \left[0, 1\right].
\end{equation}
Then we must have
\begin{equation}
2 \int_{\utri} \det(D\varphi) \left[F \circ \varphi\right] \, dt \, ds =
\oint_{b_0 \cup b_1 \cup b_2} H \, dy - V \, dx
\end{equation}
for antiderivatives that satisfy \(H_x = V_y = F\).

When \(\det(D\varphi) > 0\), this is just the change of variables
formula combined with Green's theorem.
\end{theorem}

\begin{proof}
Let \(x(s, t)\) and \(y(s, t)\) be the components of \(\varphi\). Define
\begin{equation}
\Delta S = H(x, y) y_s - V(x, y) x_s \quad \text{and} \quad
\Delta T = H(x, y) y_t - V(x, y) x_t.
\end{equation}
On the unit triangle \(\utri\), Green's theorem gives
\begin{equation}\label{eq:basic-greens}
\int_{\mathcal{U}} \left[\partial_s \Delta T -
  \partial_t \Delta S\right] \, dV =
\oint_{\partial \mathcal{U}} \Delta S \, ds + \Delta T \, dt.
\end{equation}
The boundary \(\partial \utri\) splits into the bottom edge \(E_0\),
hypotenuse \(E_1\) and left edge \(E_2\).

Since
\begin{equation}
E_0 = \left\{ \left[ \begin{array}{c} r \\ 0 \end{array}\right] \mid
  r \in \left[0, 1\right] \right\}
\end{equation}
we take \(\varphi(r, 0) = b_0(r)\) hence
\begin{equation}
dx = x_s \, dr, dy = y_s \, dr \Longrightarrow
H dx - V dy = \Delta S \, dr.
\end{equation}
We also have \(ds = dr\) and \(dt = 0\) due to the
parameterization, thus
\begin{equation}
\int_{E_0} \Delta S \, ds + \Delta T \, dt =
  \int_{r = 0}^{r = 1} \Delta S \, dr = \int_{b_0} H \, dx - V \, dy.
\end{equation}
We can similarly verify that
\(\int_{E_j} \Delta S \, ds + \Delta T \, dt = \int_{b_j} H \, dx - V \, dy\)
for the other two edges. Combining this with~\eqref{eq:basic-greens}
we have
\begin{equation}
\int_{\mathcal{U}} \left[\partial_s \Delta T -
  \partial_t \Delta S\right] \, dV =
\oint_{b_0 \cup b_1 \cup b_2} H \, dx - V \, dy.
\end{equation}
To complete the proof, we need
\begin{equation}
\int_{\mathcal{U}} \left[\partial_s \Delta T -
  \partial_t \Delta S\right] \, dV =
2 \int_{\mathcal{U}} \det(D\varphi) \left[F \circ \varphi\right] \, dV
\end{equation}
but one can show directly that
\begin{equation}
\partial_s \Delta T - \partial_t \Delta S =
  2 \left(x_s y_t - x_t y_s\right) F(x, y) =
  2 \det(D\varphi) \left[F \circ \varphi\right]. \tag*{\qedhere}
\end{equation}
\end{proof}

\end{document}